%% file: main.tex
\documentclass[11pt,letterpaper]{article}
\usepackage{float, graphicx}

\usepackage[]{epsfig}
\usepackage{amsmath, amsthm, amssymb,}
\usepackage{aliascnt}
\usepackage{epsfig}
\usepackage{verbatim}
\usepackage{multicol}
\usepackage{url}
\usepackage{latexsym}
\usepackage{mathrsfs}
\usepackage[colorlinks, bookmarks=true, citecolor=blue]{hyperref}
\usepackage{amsmath}
\usepackage{enumerate}
\usepackage{bm}
\usepackage{enumitem}
\usepackage{tikz}
\usepackage{cleveref}

\usepackage{cite}
\usepackage{color}
\usepackage{xcolor}

\usepackage{authblk}

\usepackage[margin=0.8in]{geometry}

\makeatletter

\def\@settitle{\begin{center}%
    \bfseries
 \normalfont\LARGE\@title
  \end{center}%
}
\def\@setauthors{\begin{center}%
 \normalsize\@author
  \end{center}%
}

\makeatother

\numberwithin{equation}{section}

\newcommand{\sfT}{{\mathsf T}}

\renewcommand{\cal}{\mathcal}

\newcommand{\cF}{{\cal F}}

\newcommand{\fc}{{\mathfrak c}}

\newcommand{\rd}{{\rm d}}

\newcommand{\ri}{\mathbf{i}}

\newcommand{\R}{\mathbb{R}}

\newcommand{\bC}{{\mathbb C}}

\newcommand{\bH}{\mathbb{H}}

\newcommand{\bR}{{\mathbb R}}

\newcommand{\bZ}{\mathbb{Z}}

\newcommand{\eps}{\varepsilon}

\DeclareMathOperator{\semci}{sc}

\DeclareMathOperator{\Tr}{Tr}

\DeclareMathOperator{\supp}{supp}
\DeclareMathOperator{\spec}{spec}
\DeclareMathOperator{\dist}{dist}

\DeclareMathOperator{\oo}{o}

\renewcommand{\Re}{\mathop{\mathrm{Re}}}
\renewcommand{\Im}{\mathop{\mathrm{Im}}}

\renewcommand{\leq}{\leqslant}
\renewcommand{\geq}{\geqslant}

\newcommand{\beq}{\begin{equation}}
\newcommand{\eeq}{\end{equation}}
\theoremstyle{plain}

\newtheorem{theorem}{Theorem}[section]
\newtheorem*{theorem*}{Theorem}

\newaliascnt{lemma}{theorem}
\newtheorem{lemma}[lemma]{Lemma}
\aliascntresetthe{lemma}
\newtheorem*{lemma*}{Lemma}

\newaliascnt{corollary}{theorem}

\aliascntresetthe{corollary}
\newtheorem*{corollary*}{Corollary}

\newaliascnt{proposition}{theorem}
\newtheorem{proposition}[proposition]{Proposition}
\aliascntresetthe{proposition}
\newtheorem*{proposition*}{Proposition}

\newaliascnt{assumption}{theorem}

\aliascntresetthe{assumption}
\newtheorem*{assumption*}{Assumption}

\newaliascnt{claim}{theorem}

\aliascntresetthe{claim}

\newaliascnt{conjecture}{theorem}

\aliascntresetthe{conjecture}

\newaliascnt{definition}{theorem}

\aliascntresetthe{definition}
\newtheorem*{definition*}{Definition}

\newaliascnt{example}{theorem}

\aliascntresetthe{example}
\newtheorem*{example*}{Example}

\newaliascnt{remark}{theorem}
\newtheorem{remark}[remark]{Remark}
\aliascntresetthe{remark}
\newtheorem*{remark*}{Remark}

\newaliascnt{remarks}{theorem}

\aliascntresetthe{remarks}
\newtheorem*{remarks*}{Remarks}

\crefname{assumption}{assumption}{assumptions}
\Crefname{assumption}{Assumption}{Assumptions}

\crefname{theorem}{theorem}{theorems}
\Crefname{theorem}{Theorem}{Theorems}
\crefname{lemma}{lemma}{lemmas}
\Crefname{lemma}{Lemma}{Lemmas}
\crefname{proposition}{proposition}{propositions}
\Crefname{proposition}{Proposition}{Propositions}
\crefname{corollary}{corollary}{corollaries}
\Crefname{corollary}{Corollary}{Corollaries}
\crefname{claim}{claim}{claims}
\Crefname{claim}{Claim}{Claims}
\crefname{conjecture}{conjecture}{conjectures}
\Crefname{conjecture}{Conjecture}{Conjectures}
\crefname{definition}{definition}{definitions}
\Crefname{definition}{Definition}{Definitions}
\crefname{example}{example}{examples}
\Crefname{example}{Example}{Examples}
\crefname{remark}{remark}{remarks}
\Crefname{remark}{Remark}{Remarks}

\title{Support of Dyson Brownian Motion}
\author{Jiaoyang Huang  \quad\quad\quad   Shengjing Xu }

\date{January 2026}

\begin{document}

\maketitle
\begin{abstract}
We consider \(\beta\)-Dyson Brownian motion, with \(\beta\geq 1\),
started from a deterministic configuration with uniformly bounded
support. Let \(\mu_t\) be the semicircular free-convolution flow issued
from the initial empirical measure, and set
$
S_t=\operatorname{supp}\mu_t.
$
For every fixed \(\sf T,\varepsilon>0\), with probability at least
$1-Ce^{-(\log n)^2}$,
every particle remains within an \(\varepsilon\)-neighborhood of
\(S_t\) for all \(0\leq t\leq \sf T\). The result holds from time zero,
requires no regularity assumption at the initial spectral edges, and
applies to multi-cut supports with macroscopic interior gaps.

A key ingredient is a deterministic local resolvent exclusion
principle: an \(\oo((n\eta)^{-1})\) comparison of Stieltjes transforms on
a complex disc above a real point separated from the reference support
excludes eigenvalues from the corresponding real interval. This gives a
model-independent mechanism for converting local resolvent estimates
into spectral confinement.
\end{abstract}

%\tableofcontents

\input{intro}

\input{main_results}

\input{Random_matrix_example}

\input{DBM}

\bibliography{ref}
\bibliographystyle{abbrv}

\end{document}

%% file: intro.tex
\section{Introduction}

In 1962, Dyson introduced matrix-valued Brownian motion and observed that
its eigenvalues evolve as an interacting diffusion \cite{MR0148397}.  Its
general-$\beta$ extension is the $\beta$-Dyson Brownian motion
\begin{equation}\label{e:DBM}
  \rd\lambda_i(t)
  =
  \left(\frac{2}{\beta n}\right)^{1/2}\rd B_i(t)
  +\frac1n\sum_{j\ne i}
  \frac{\rd t}{\lambda_i(t)-\lambda_j(t)},
  \qquad 1\leq i\leq n,
\end{equation}
where $\beta\geq1$ and $B_1,\ldots,B_n$ are independent Brownian
motions.  We consider a solution started from a deterministic, strictly
ordered configuration
\[
  \lambda_1(0)>\cdots>\lambda_n(0).
\]
For $\beta=1$ and $\beta=2$, \eqref{e:DBM} gives, under the usual
normalization, the eigenvalue processes of real symmetric and complex
Hermitian matrix Brownian motion, respectively.  Dyson Brownian motion is
a basic model for a one-dimensional interacting particle system with
singular repulsion and has played a central role in the proof of random
matrix universality; see \cite{erdHos2011universality,erdHos2017dynamical}.

We study the location of all particles relative to the deterministic
free-convolution flow generated by the initial empirical measure.  Assume
that the initial configuration is contained in a fixed compact interval,
uniformly in $n$, and set
\[
  \mu_0:=\frac1n\sum_{i=1}^n\delta_{\lambda_i(0)},
  \qquad
  \mu_t:=\mu_0\boxplus\mu_{\mathrm{sc}}^{(t)},
  \qquad
  S_t:=\operatorname{supp}\mu_t,
\]
where $\mu_{\mathrm{sc}}^{(t)}$ is the semicircle law of variance $t$.
The measure $\mu_t$ is the deterministic semicircular free-convolution
flow issued from $\mu_0$ \cite{FCSD}.  Writing
\[
  \Lambda_n(t):=\{\lambda_1(t),\ldots,\lambda_n(t)\},
\]
our main result states that,

\begin{theorem}
\label{t:DBM_support}
We assume that there exists a constant $K_0>0$
\begin{equation}
  \operatorname{supp}\mu_0
  \subset[-K_0,K_0]
  \label{eq:initial-support-bound}
\end{equation}
For every \(\varepsilon>0\) and \(\sfT>0\), there exist
\(
  C=C(\varepsilon,\sfT,K_0,\beta)>0
\)
such that, for $n$ large enough,
\begin{equation}
   \mathbb P\left(
    \Lambda_n(t)\subset S_t^{[\varepsilon]}
    \text{ for every }0\leq t\leq \sfT
  \right)
  \geq
  1-Ce^{-(\log n)^2}.
  \label{eq:time-uniform-multicut-confinement}
\end{equation}
where $S_t^{[\varepsilon]}$ denotes the closed $\varepsilon$-neighborhood
of $S_t$
\[
  S_t^{[\varepsilon]}
  :=
  \left\{
    x\in\mathbb R:
    \operatorname{dist}(x,S_t)\le\varepsilon
  \right\}.
\]
\end{theorem}

The result starts at time zero, when
$S_0=\Lambda_n(0)$, and follows the support throughout the entire interval
$[0,\sfT]$.  It requires neither convergence of $\mu_0$ to a limiting
measure nor any regularity assumption on its edges.  In particular, the
support may have several components, possibly depending on $n$, and
\eqref{eq:time-uniform-multicut-confinement} excludes particles both outside the
outermost components and inside every macroscopic interior gap.  The
conclusion is one-sided: we do not claim that every point of $S_t$ lies
close to a particle.

\paragraph{Comparison with fixed-time deformed Wigner matrices.}
For the classical values $\beta=1,2$, a fixed-time marginal of additive
matrix Dyson Brownian motion is, after normalization, a deterministic
matrix plus a Gaussian Wigner matrix. Capitaine, Donati-Martin, F\'eral,
and F\'evrier studied the more general deformed Wigner model
\[
  M_N=\frac1{\sqrt N}W_N+A_N
\]
and proved spectral inclusion relative to the exact finite-$N$ free
convolution, together with an exact separation theorem for admissible
gaps \cite{capitaine2011free}. In particular, under their assumptions,
for every fixed $\varepsilon>0$, almost surely for all sufficiently large
$N$,
\[
  \operatorname{Spec}(M_N)
  \subset
  \bigl(
    \operatorname{supp}
    (\mu_{\mathrm{sc}}^{(\sigma)}\boxplus\mu_{A_N})
  \bigr)^{[\varepsilon]},
\]
where $\mu_{\mathrm{sc}}^{(\sigma)}$ is the semicircle law determined by
the variance of the Wigner entries.

A significant restriction in their framework is the fixed finite-spike
assumption. The empirical measure of $A_N$ is required to converge to a
fixed compactly supported measure $\nu$; all but a fixed number of the
eigenvalues of $A_N$ must approach $\operatorname{supp}\nu$ uniformly;
and the remaining eigenvalues are fixed spikes with fixed
multiplicities. Thus, $A_N$ itself may have full rank, but only finitely
many of its eigenvalues may remain macroscopically separated from the
limiting bulk. Their assumptions therefore do not cover a growing number
of separated eigenvalues or clusters, arbitrary $N$-dependent spike
locations, or deterministic configurations for which no limiting bulk
measure is prescribed.

More broadly, the exclusion of such exceptional eigenvalues is the main
spectral content of strong convergence. Convergence of empirical spectral
measures does not detect finitely many outliers, whereas strong
convergence controls operator norms of noncommutative polynomials and
thereby rules out spectrum away from the corresponding limiting support.
A recent series of works has developed a polynomial method for proving
this stronger form of convergence. The central observation is that
quantitative first-order expansions of expected traces contain enough
information to exclude outliers and establish strong convergence
\cite{chen2025newapproachstrongconvergence,
chen2024newapproachstrongconvergence}. This approach has also led to new
strong-convergence results beyond the classical random matrix ensembles;
see, for example,
\cite{magee2025strongconvergenceuniformlyrandom} and the survey
\cite{vanhandel2025strongconvergencephenomenon}. Although these results
are generally formulated relative to a limiting spectrum, rather than
the exact finite-$N$ free-convolution support appearing above, they
provide a flexible general mechanism for converting quantitative trace
asymptotics into the exclusion of spectral outliers.

\paragraph{Comparison with microscopic rigidity.}
Characteristic-flow local laws and rigidity estimates for Dyson Brownian
motion with general $\beta$ were developed in \cite{HuangLandon2019}.
Much sharper results are available near a specified regular outer edge
\cite{aggarwal2024edge,landon2017convergence,AdhikariHuang2020}.  Under suitable
square-root assumptions on the initial density, these works obtain
microscopic rigidity on the $n^{-2/3}$ scale, typically after a positive
waiting time.  Such results are stronger than ours near a regular edge,
but they do not directly yield a single time-uniform statement for the
entire, possibly multi-cut, support of an arbitrary deterministic initial
configuration.  Our result is therefore complementary: it is macroscopic,
but global in both time and support geometry.

\paragraph{A deterministic local exclusion principle.}
The main technical input is a deterministic criterion that converts a
local resolvent estimate into spectral exclusion.  Let $Y_n$ be Hermitian,
let $\mu$ be a compactly supported reference measure, and suppose that
$u$ is at distance at least $\varepsilon$ from $\operatorname{supp}\mu$.
If $\eta\leq n^{-c}\varepsilon$ and
\[
  \sup_{z\in B(u+\mathrm i\eta,\eta/2)}
  \bigl|m_{Y_n}(z)-m_\mu(z)\bigr|
  \leq
  \frac{\delta_n}{n\eta},
  \qquad
  \delta_n\longrightarrow0,
\]
then, for all sufficiently large $n$, $Y_n$ has no eigenvalue in
$[u-\eta,u+\eta]$.  The proof uses Schwarz symmetry and higher-order
Cauchy estimates to extract a positive Poisson-type kernel from the
Stieltjes transform.  This criterion is model independent: once the local
resolvent comparison is known, no further spectral argument is needed.

To apply it to Dyson Brownian motion, we place small complex discs above
the boundary of the forbidden region and transport them backward along
the characteristics of the free-convolution equation.  A stopping time
records the first failure of the resolvent comparison.  Before that time,
the deterministic criterion prevents any particle from crossing the
moving barriers.  It\^o's formula yields an evolution equation for the
resolvent error along each characteristic, and the martingale and
finite-$n$ drift terms are controlled by the
Burkholder--Davis--Gundy and Gr\"onwall inequalities
\cite{DFIM,PTIIM}.  Macroscopic separation from the full support provides
uniform estimates even when the support has several components, closing
the bootstrap and proving \eqref{eq:time-uniform-multicut-confinement}.

\paragraph{Organization.}
\Cref{s:exclusion} establishes the
deterministic local exclusion principle. \Cref{s:dbm} applies the deterministic
local principle to Dyson Brownian motion and proves time-uniform
multi-cut spectral confinement from the initial time.

%% file: main_results.tex
\section{Stieltjes-transform  for outliers}
\label{s:exclusion}

Let \(Y_n\) be a deterministic Hermitian \(n\times n\) matrix, its eigenvalues are real; denote them (counted with algebraic multiplicity) by
\[
\lambda^{(n)}_1 \ge \lambda^{(n)}_2 \ge \cdots \ge \lambda^{(n)}_n \in \mathbb R .
\]
Define the empirical spectral measure of $Y_n$ by
\[
\mu_{Y_n}:=\frac{1}{n}\sum_{i=1}^n \delta_{\lambda^{(n)}_i}.
\]
Its Stieltjes transform is
\[
m_{Y_n}(z):=\int_{\mathbb R}\frac{1}{t-z}\,\mu_{Y_n}(dt)
=\frac{1}{n}\Tr\bigl((Y_n-zI)^{-1}\bigr),
\qquad z\in\mathbb C\setminus\mathbb R.
\]
For a probability measure $\mu$ on $\mathbb R$, we define the Stieltjes transform of $\mu$ 
\[
m_{\mu}(z):=\int_{\mathbb R}\frac{1}{t-z}\,\mu(dt),
\qquad z\in\mathbb C\setminus\mathbb R.
\]

We have the following deterministic resolvent comparison, which can be used to rule out outlier eigenvalues. 
\begin{theorem}[Deterministic resolvent comparison]\label{c:deterministic_comparison}
Let \(Y_n\) be a deterministic Hermitian \(n\times n\) matrix, and there exists a constant $K>0$, such that 
$\supp(\mu_{Y_n})\subset [-K,K]$. Let \(\mu\) be a
compactly supported probability measure. Assume that there exist a constant
\(\fc>0\), scales \(\varepsilon=\varepsilon_n\geq n^{-1}\) and
\(\eta=\eta_n>0\), and an error parameter \(\delta_n>0\) such that
\begin{equation}\label{eq:det_eta_eps_relation}
\eta\leq n^{-\fc}\varepsilon,
\qquad
\delta_n\to 0.
\end{equation}
Let \(u=u_n\in\bR\) satisfy \(\dist(u,\supp\mu)\geq\varepsilon\), and assume
\begin{equation}\label{eq:deterministic_resolvent_bound}
\sup_{z\in B(u+\ri\eta,\eta/2)}
\bigl|m_{Y_n}(z)-m_{\mu}(z)\bigr|
\leq \frac{\delta_n}{n\eta}.
\end{equation}
Then, for all sufficiently large \(n\),
\[
\lambda_i^{(n)}\notin[u-\eta,u+\eta],
\qquad 1\leq i\leq n.
\]
\end{theorem}

\subsection{Proof of \Cref{c:deterministic_comparison}}
We first record the Cauchy estimate used in the deterministic comparison.
\begin{lemma}\label{l:wirtinger_calculus}
Let \(z_0=u+\ri\eta\), \(w_0=u-\ri\eta\), and suppose that \(f\) is analytic in
neighborhoods of the closed discs \(B(z_0,\eta/2)\) and \(B(w_0,\eta/2)\). Define
\[
H_f(z,w):=\frac{f(z)-f(w)}{z-w}.
\]
For all integers \(a,b\geq0\),
\begin{equation}\label{eq:lemma-cauchy}
\bigl|\partial_z^a\partial_w^bH_f(z_0,w_0)\bigr|
\leq C_{a,b}\eta^{-(a+b+1)}
\sup_{B(z_0,\eta/2)\cup B(w_0,\eta/2)}|f|,
\end{equation}
where
\[
C_{a,b}:=2^{a+b+1}a!b!\bigl(3^{-(a+1)}+3^{-(b+1)}\bigr).
\]
In particular,
\begin{equation}\label{eq:lemma-cauchy-special}
\bigl|\partial_z^{k-1}\partial_w^kH_f(z_0,w_0)\bigr|
\leq\frac{2^{2k+2}}{3^{k+1}}(k-1)!k!\,\eta^{-2k}
\sup_{B(z_0,\eta/2)\cup B(w_0,\eta/2)}|f|.
\end{equation}
\end{lemma}

\begin{proof}
With \(r=\eta/2\) and
\(\omega=\partial B(z_0,r)\cup\partial B(w_0,r)\), both circles oriented
counterclockwise, Cauchy's formula gives
\[
\partial_z^a\partial_w^bH_f(z_0,w_0)
=\frac{a!b!}{2\pi\ri}\oint_\omega
\frac{f(\zeta)\,\rd\zeta}
{(\zeta-z_0)^{a+1}(\zeta-w_0)^{b+1}}.
\]
On the first circle, \(|\zeta-z_0|=r\) and \(|\zeta-w_0|\geq3r\); the roles are
reversed on the second circle. Estimating the two integrals proves
\eqref{eq:lemma-cauchy}, and \(a=k-1,b=k\) gives
\eqref{eq:lemma-cauchy-special}.
\end{proof}

\begin{proof}[Proof of \Cref{c:deterministic_comparison}]
Set \(z_0=u+\ri\eta\), \(w_0=u-\ri\eta\), and
\(\Delta_n=m_{Y_n}-m_\mu\). Then
\begin{equation}\label{eq:det-local-H}
H_{\Delta_n}(z,w)
=\frac1n\sum_{i=1}^n\frac1{(\lambda_i^{(n)}-z)(\lambda_i^{(n)}-w)}
-\int_\bR\frac{\mu(\rd x)}{(x-z)(x-w)}.
\end{equation}
For every integer \(k\geq1\),
\begin{equation}\label{eq:det-local-derivative}
\partial_z^{k-1}\partial_w^kH_{\Delta_n}(z,w)
=k!(k-1)!\left[
\frac1n\sum_{i=1}^n
\frac1{(\lambda_i^{(n)}-z)^k(\lambda_i^{(n)}-w)^{k+1}}
-\int_\bR\frac{\mu(\rd x)}{(x-z)^k(x-w)^{k+1}}
\right].
\end{equation}
Since
\[
\Im\frac1{(t-z_0)^k(t-w_0)^{k+1}}
=-\frac{\eta}{((t-u)^2+\eta^2)^{k+1}},
\]
we have
\begin{equation}\label{eq:det-local-imaginary}
-\Im\partial_z^{k-1}\partial_w^kH_{\Delta_n}(z_0,w_0)
=k!(k-1)!\left[
\frac1n\sum_{i=1}^n
\frac{\eta}{((\lambda_i^{(n)}-u)^2+\eta^2)^{k+1}}
-\int_\bR\frac{\eta\,\mu(\rd x)}{((x-u)^2+\eta^2)^{k+1}}
\right].
\end{equation}

Suppose that some \(\lambda_i^{(n)}\in[u-\eta,u+\eta]\). The corresponding
summand is at least \(2^{-(k+1)}\eta^{-(2k+1)}\), while the last integral in
\eqref{eq:det-local-imaginary} is at most \(\eta\varepsilon^{-2k-2}\).
Schwarz symmetry extends \eqref{eq:deterministic_resolvent_bound} to the disc
around \(w_0\), so \Cref{l:wirtinger_calculus} gives
\begin{equation}\label{eq:det-local-contradiction}
2^{-(k+1)}
\leq C_k\delta_n+n\left(\frac{\eta}{\varepsilon}\right)^{2k+2}.
\end{equation}
By \eqref{eq:det_eta_eps_relation},
\begin{equation}\label{eq:det-local-scale}
n\left(\frac{\eta}{\varepsilon}\right)^{2k+2}
\leq n^{1-\fc(2k+2)}.
\end{equation}
Choose \(k\) so that \(1-\fc(2k+2)<0\). The right-hand side of
\eqref{eq:det-local-contradiction} tends to zero, contradicting its positive
left-hand side. Therefore, the claim follows.
\end{proof}

%% file: Random_matrix_example.tex
\subsection{Comparison with other resolvent exclusion arguments}
\label{s:comparison-resolvent-arguments}

It is useful to compare the criterion in \eqref{eq:deterministic_resolvent_bound} with the standard argument based on the imaginary part of the resolvent in random matrix theory. The latter is  common in the random matrix literature, but it requires the spectral scale \(\eta\) to be sufficiently small.

\begin{proposition}[Classical one-point criterion]\label{prop:classical_one_point}
Let \(z=u+\ri\eta\in\bH\), and assume that
$
\dist\bigl(u,\supp(\mu)\bigr)\ge \varepsilon.
$
Suppose that
\begin{equation}\label{eq:classical_local_law}
|m_{Y_n}(z)-m_{\mu}(z)|\ll \frac{1}{n\eta},\quad 
\Im m_{\mu}(z)\ll \frac{1}{n\eta},
\end{equation}
then
$
  \operatorname{spec}(Y_n)\cap[u-\eta,u+\eta]=\varnothing.
$
\end{proposition}

\begin{proof}
Assume, by contradiction, that \(\lambda_j^{(n)}\in [u-\eta,u+\eta]\) for some \(j\). Then
\[
\Im m_{Y_n}(z)
=
\frac1n\sum_{i=1}^n \frac{\eta}{(\lambda_i^{(n)}-u)^2+\eta^2}
\geq
\frac1n\frac{\eta}{(\lambda_j^{(n)}-u)^2+\eta^2}
\ge
\frac{1}{2n\eta}.
\]
On the other hand, by \eqref{eq:classical_local_law}
\[
\Im m_{Y_n}(z)
\le
|m_{Y_n}(z)-m_{\mu}(z)|+\Im m_{\mu}(z)
\ll
\frac{1}{n\eta},
\]
a contradiction. Therefore no eigenvalue can lie in \([u-\eta,u+\eta]\).
\end{proof}

The point is that the one-point argument only works when \(\Im m_{\mu}(z)\) is itself much smaller than \((n\eta)^{-1}\). A crude bound gives
\[
\Im m_{\mu}(u+\ri\eta)
=
\int_{\R}\frac{\eta}{(x-u)^2+\eta^2}\,\mu(\rd x)
\le
\frac{\eta}{\varepsilon^2},
\]
since \(\dist(u,\supp\mu)\ge \varepsilon\). Hence a sufficient condition for the second statement in \eqref{eq:classical_local_law} is
\[
\frac{\eta}{\varepsilon^2}\ll \frac{1}{n\eta},
\qquad\text{that is}\quad
\eta \ll \frac{\varepsilon}{\sqrt n}.
\]
Thus the classical method already requires \(\eta\) to be at most of order \(n^{-1/2}\) when \(\varepsilon\asymp 1\).

Near a regular edge one can use the sharper estimate
\[
\Im m_{\mu}(E+\kappa+\ri\eta)\asymp \frac{\eta}{\sqrt{\kappa+\eta}},
\]
so the one-point argument requires
\[
\frac{\eta}{\sqrt{\kappa+\eta}} \ll \frac{1}{n\eta},
\qquad\text{equivalently}\qquad
n\eta^2 \ll \sqrt{\kappa+\eta}.
\]
When \(\kappa\gg \eta\), this becomes
\(
\eta \ll n^{-1/2}\kappa^{1/4}.
\)
In particular, if \(\kappa\asymp 1\), then one again needs \(\eta\ll n^{-1/2}\). The criterion in \eqref{eq:deterministic_resolvent_bound} does not rely on this one-point scale condition.

The spectral mechanism underlying
\Cref{c:deterministic_comparison} is closely related to the argument in
\cite[Section~5.5]{AGZ}. In their independent GUE setting,
\cite[Lemma~5.5.4]{AGZ} gives a pointwise comparison between the
expected finite-dimensional Stieltjes transform and its free limit on
a bounded strip, down to a polynomial imaginary scale.
\begin{equation}
  \left|
    \mathbb E m_{Y_N}(z)-m_{\mu_y}(z)
  \right|
  \le
  \frac{C}{N^2(\Im z)^{C'}},
  \qquad
  N^{-c'}\le \Im z\le c.
  \label{eq:agz-pointwise-resolvent}
\end{equation}
Using an almost-analytic functional calculus,
\cite[Lemma~5.5.5]{AGZ} then obtains
\[
  \left|
    \mathbb E\int_{\mathbb R}\phi\,\rd\mu_{Y_N}
    -
    \int_{\mathbb R}\phi\,\rd\mu_y
  \right|
  \le
  \frac{C(\phi)}{N^2}
\]
for every fixed smooth compactly supported function \(\phi\) vanishing
on \(\operatorname{spec}(y)\). After choosing \(\phi\ge0\) on a
forbidden region, positivity and Gaussian concentration yield
eventual almost-sure $\operatorname{spec}(Y_n) \to \operatorname{spec}(y)$.

Thus, the two arguments use the same basic principle: Stieltjes difference is converted into a positive spectral statistic that
cannot remain small when eigenvalues lie in some domain.
The form of the analytic input is different. The estimate
in \cite{AGZ} is global and averaged: it controls an expected
Stieltjes transform on an entire strip and consequently a large class
of smooth test functions. By contrast,
\Cref{c:deterministic_comparison} uses a local
implication to prevent eigenvalues lie in some small intervals. 

In the next section, \Cref{c:deterministic_comparison} can directly apply to the case of multi-cut Dyson Brownian
motion. Since the support of classical measure
\(
  S_t=\operatorname{supp}
  \bigl(\mu_0\boxplus\mu_{\rm sc}^{(t)}\bigr)
\)
moves with time and have several connected components. We will use \Cref{c:deterministic_comparison} to control each connected component for any $t$. And therefore, show a time-uniform spectral confinement statement.

%% file: DBM.tex
\section{Support of Dyson Brownian motion}
\label{s:dbm}

In this section we prove the main result \Cref{t:DBM_support}.
\subsection{Free Convolution with Semicircle Distributions}
\label{TransformConvolution}

    We recall the basic facts about Stieltjes transforms and free convolution
with a semicircle distribution. For a finite measure $\mu$ on
$\mathbb R$, its Stieltjes transform is
\begin{flalign}
		\label{mz0} 
		m(z) = \displaystyle\int_{-\infty}^{\infty} \displaystyle\frac{\mu (\rd x)}{x-z}.
	\end{flalign}
We have the following standard estimates.
\begin{lemma}\label{l:STproperty}
Let $m(z)=m_\mu(z)$ be the Stieltjes transform of a finite measure $\mu$ with $A=\mu(\bR)$. For any integer $p\geq 1$, we denote its $p$-th derivative by $m^{(p)}(z)$. Then, 
\begin{align}\label{e:stbb}
|m(z)|\leq \frac{A}{\dist(z,\supp(\mu))},\quad
\quad |m'(z)|\leq \frac{\Im[m(z)]}{\Im[z]}, 
\quad |m^{(p)}(z)|\leq \frac{p!A}{\dist(z,\supp(\mu))^{p+1}}
\end{align}
\end{lemma}

	The \emph{semicircle distribution} is a probability measure $\mu_{\semci}$ whose density $\varrho_{\semci} : \mathbb{R} \rightarrow \mathbb{R}_{\ge 0}$ with respect to the Lebesgue measure is given by 
	\begin{flalign}
		\label{rho1} 
		\varrho_{\semci} (x) = \displaystyle\frac{(4-x^2)^{1/2}}{2\pi} \cdot \textbf{1}_{x \in [-2, 2]}, \qquad \text{for all $x \in \mathbb{R}$}. 
	\end{flalign}
	
	\noindent For any $t > 0$, we denote by $\varrho_{\semci}^{(t)}$ and $\mu_{\semci}^{(t)}$ the rescaled semicircle density and probability measure, respectively.
	\begin{flalign}
		\label{rhosct}
		\varrho_{\semci}^{(t)} (x) = t^{-1/2} \varrho_{\semci} (t^{-1/2} x); \qquad \mu_{\semci}^{(t)} = \varrho_{\semci}^{(t)} (x) d x.
	\end{flalign}
For any finite measure $\mu$ on
$\mathbb R$, denote its Stieltjes transform $m(z)$. For any $t > 0$, define the function $F_t = z - t m(z): \mathbb{H} \rightarrow \mathbb{C}$ and the set $\Lambda_t \subseteq \mathbb{H}$ by 
	\begin{align}\label{mtlambdat} 
		 \Lambda_t = \Big\{ z \in \mathbb{H} : \Im \big( z - tm (z) \big) > 0 \Big\} = \Bigg\{ z \in \mathbb{H} : \int_{-\infty}^{\infty} \frac{\mu(\rd x)}{|z-x|^2} < \frac{1}{t} \Bigg\}.
	\end{align}

\begin{lemma}[{\cite[Lemma 4]{FCSD}}] \label{mz} 
$F_t$ is a homeomorphism from $\overline{\Lambda}_t$ to $\overline{\mathbb{H}}$. Moreover, it is a holomorphic map from $\Lambda_t$ to $\mathbb{H}$ and a bijection from $\partial \Lambda_t$ to $\mathbb{R}$. 
\end{lemma} 

	For any $t \ge 0$, define $m_t: \mathbb{H} \rightarrow \mathbb{H}$ as
    \begin{flalign}
		\label{mt} 
		m_t \big( z - t m_0 (z) \big) = m (z), \qquad \text{for any $z \in \Lambda_t$}.
	\end{flalign}
    Note that $m_0 (z) = m(z)$.
	By \cite[Proposition 2]{FCSD}, $m_t$ is the Stieltjes transform of a probability measure $\mu_t$. We call this measure to be the \emph{free convolution} between $\mu$ and $\mu_{\semci}^{(t)}$, which denote by $\mu_t = \mu \boxplus \mu_{\semci}^{(t)}$. By \cite[Corollary 2]{FCSD}, $\mu_t$ has a density $\varrho_t = \varrho_t^{\mu} : \mathbb{R} \rightarrow \mathbb{R}_{\ge 0}$ with respect to Lebesgue measure for $t > 0$.	

\subsection{Characteristic flow}\label{s:Cf}
Recall the $m_t(z)$ in \eqref{mt},
we then introduce the characteristic flow:
\begin{align}\label{e:flow}
z_t(u)=u-tm_0(u)=q_t(u)+\ri\eta_t(u),\quad u\in \Lambda_t,
\end{align}
satisfying the property that 
\begin{flalign}\label{e:m0=ms}
	\partial_tm_t (z_t)=0 \qquad \text{$t \ge 0$},
\end{flalign} 
When the context is clear, we will simply write $z_t(u), q_t(u), \eta_t(u)$ as $z_t, q_t, \eta_t$. 
Thus, along the char flow, we have 
\begin{align}\label{e:etas}
z_s=z_t+(t-s)m_t(z_t),\quad  \eta_s=\eta_t+(t-s) \Im[m_t(z_t)]\geq \eta_t,
\end{align}
where the last inequality follows from $\Im[m_t(z_t)]=\Im[m_0(u)]\geq 0$.

For each \(s>0\), write
\[
S_s=\supp(\mu_s)
=
\bigcup_{j=1}^{M(s)}[a_j(s),b_j(s)],
\qquad
a_1(s)<b_1(s)<\cdots<a_{M(s)}(s)<b_{M(s)}(s).
\]
By \cite{FCSD},
\(M(s)\) is non-increasing in \(s\), since different components of the
support may merge. Whenever we differentiate individual edges, we work
on an open interval \(I\subset(0,T]\) on which \(M(s)\equiv M\) is
constant and we call all edges are regular.

\begin{lemma}\label{c:edge}
    For each $1\le j\le M$, the regular edges of $\supp(\mu_s)$ satisfy
\begin{equation} \label{eq:physical-edge-speeds}
  a_j'(s)=-m_s(a_j(s)),
  \qquad
  b_j'(s)=-m_s(b_j(s)).
\end{equation}
\end{lemma}

\begin{proof}
By \cite{FCSD}, the support of the semicircular free convolution
identifies every regular edge with a nondegenerate real critical point
of $F_s$ in \eqref{mtlambdat}. For each $1\le j\le M$, there are real-valued functions
$\alpha_j(s)$ and $\beta_j(s)$ such that
\begin{equation}
  s m_0'(\alpha_j(s))=s m_0'(\beta_j(s))=1,
  \qquad
  m_0''(\alpha_j(s))>0,
  \qquad
  m_0''(\beta_j(s))<0.
  \label{eq:regular-edge-critical-points}
\end{equation}
Since
$$
  a_j(s)=F_s(\alpha_j(s)),
  \qquad
  b_j(s)=F_s(\beta_j(s)),
$$
differentiating
yields
\begin{align*}
  a_j'(s)
  &=\alpha_j'(s)-m_0(\alpha_j(s))
    -s m_0'(\alpha_j(s))\alpha_j'(s)
   =-m_0(\alpha_j(s))=-m_s(a_j(s)).
\end{align*}
This proves
\eqref{eq:physical-edge-speeds}.
\end{proof}

\begin{remark}[Merging of adjacent edges]\label{r:edge-merging}
Lemma~\ref{c:edge} is local in time: it applies on intervals on which
the number of support components is constant and the corresponding
critical points are nondegenerate.
When two adjacent interior edges merge at a time
\(s_\mathrm{c}\), namely
\[
b_j(s),\,a_{j+1}(s)
\longrightarrow x_\mathrm{c},
\qquad s\uparrow s_\mathrm{c}.
\]
Then their real critical preimages satisfy
\[
\beta_j(s),\,\alpha_{j+1}(s)
\longrightarrow q_\mathrm{c},
\]
where
\begin{equation}\label{eq:merging-critical-point}
  s_\mathrm{c}m_0'(q_\mathrm{c})=1,
  \qquad
  m_0''(q_\mathrm{c})=0,
  \qquad
  x_\mathrm{c}=F_{s_\mathrm{c}}(q_\mathrm{c}).
\end{equation}
Indeed,
\[
\lim_{s\uparrow s_\mathrm{c}}b_j'(s)
=
\lim_{s\uparrow s_\mathrm{c}}a_{j+1}'(s)
=
-m_0(q_\mathrm{c}).
\]
Hence the two edges meet with the same first-order velocity. For
\(s>s_\mathrm{c}\), these two interior endpoints no longer exist as
separate edges: the two adjacent support components have merged, and
\(M(s)\) decreases by one.

Moreover, there is a cusp-merging behavior
that \[
a_{j+1}(s)-b_j(s)
\asymp
(s_\mathrm{c}-s)^{3/2},
\qquad s\uparrow s_\mathrm{c}.
\]
see
\cite[Proposition~3 and the discussion on
pp.~714--715]{FCSD}.
\end{remark}

\begin{comment}
\begin{remark}[Hausdorff continuity of the reference support]
\label{r:reference-support-continuity}
In the reduced free-product realization, let \(\mathsf a\) have law \(\mu_0\)
and let \(\mathsf s\) be a standard semicircular element free from
\(a\). Since the free-product state is faithful,
\[
  S_s=\spec\bigl(\mathsf a+\sqrt{s}\,\mathsf s\bigr),
  \qquad s\ge0.
\]
Hence the spectral variation inequality gives
\begin{equation}
  d_H(S_s,S_r)
  \le
  \left\|
  (\mathsf a+\sqrt{s}\,\mathsf s)-(\mathsf a+\sqrt{r}\,\mathsf s)
  \right\|
  \le 2|\sqrt{s}-\sqrt{r}|.
  \label{eq:reference-support-continuity}
\end{equation}
In particular, \(s\mapsto S_s\) is continuous in the Hausdorff metric
on \([0,\sfT]\), and
\[
  S_s\subset[-K_0-2\sqrt{\sfT},K_0+2\sqrt{\sfT}],
  \qquad 0\le s\le\sfT.
\]
\end{remark}
\end{comment}

\subsection{Proof of \texorpdfstring{\Cref{t:DBM_support}}{Proposition 4.1}}\label{s:prove31}
Throughout this subsection we fix a small constant \(0<\mathfrak c<1/2\), \(\delta_n \to 0\), and set
\(
\eta_*:=n^{-\mathfrak c}.
\)
Given $\beta$-Dyson Brownian motion, 
for $0\le t\le \sfT$, define the closed forbidden set as the union of distinct points
\begin{equation}
  \mathcal F_t
  :=\left\{x\in \bR:
  \operatorname{dist}(x,S_t)=\eps\right\}.
  \label{eq:forbidden-set-time-t}
\end{equation}
For $t\in[0,\sfT]$ and $x\in\mathcal F_t$, define the transported
spectral domain
\begin{equation}\label{eq:defDt-clean}
  D_s^{(t,x)}
  :=z_s\circ z_t^{-1}\left(B(x+\ri\eta_*,\eta_*/2)\right),
  \qquad 0\le s\le t,
\end{equation}
and set $D_s^{(t,x)}:=\varnothing$ for $s>t$. 
Denote the empirical Stieltjes transform of the particle configuration at time \(s\) by
\[
\widetilde m_s(z)=\frac1n\sum_{i=1}^n\frac{1}{\lambda_i(s)-z},
\qquad z\in \bH.
\]
Recall that \(m_s\) denotes the Stieltjes transform of the deterministic measure \(\mu_s\) introduced above.
Define the stopping time
\begin{equation}\label{eq:stopping-clean}
  \sigma
  :=\sf T\wedge\inf\left\{
  \begin{array}{l}
  s\ge0:\ \exists\,t\in[0,T],\ \exists\,x\in\mathcal F_t,
  \ \exists\,w\in D_s^{(t,x)}\\[1mm]
  \hspace{29mm}\text{such that }
  |\widetilde m_s(w)-m_s(w)|
  \ge\dfrac{\delta_n}{n\operatorname{Im}w}
  \end{array}
  \right\}.
\end{equation}
Since $\mu_0$ is the initial empirical measure, we have
\begin{equation}
  \widetilde m_0(z)=m_0(z),
  \qquad z\in\mathbb H.
  \label{eq:exact-initial-stieltjes}
\end{equation}

\subsubsection*{Step 1. Geometry of the transported domains}

Fix \(0<t\le \sfT\), \(x\in\mathcal F_t\), and
\(u\in D_0^{(t,x)}\). Suppose first that \(x\) lies in an interior gap
\[
  x\in \bigl(b_j(t),a_{j+1}(t)\bigr)
\]
for some index \(j\). By the definition of \(\mathcal F_t\), either
\[
  x=b_j(t)+\eps
  \qquad\text{or}\qquad
  x=a_{j+1}(t)-\eps.
\]
The exterior components are treated analogously, with only one adjacent
boundary edge. %The case \(t=0\) will be recovered at the end by the
%Hausdorff continuity in
%\Cref{r:reference-support-continuity}.

The edge calculations below are initially carried out on maximal
backward intervals on which the relevant boundary edges remain regular,
so that \Cref{c:edge} applies. As a consequence of
\Cref{lem:multicut-reference-geometry}, 
the relevant gap cannot close along the backward evolution, and hence
$b_j(r),a_{j+1}(r)$ cannot merge for \(0<r\le t\). 

The following lemma records the elementary geometry of the transported
domains.

\begin{lemma}
\label{l:geometry-clean}
There exists a constant \(C=C(\eps,\sfT,K_0)>1\) such that the following
holds for all \(0\le r\le t\le\sfT\), provided \(n\) is sufficiently
large. Fix \(x\in\mathcal F_t\) and \(u\in D_0^{(t,x)}\), and set
\[
  q_r(u):=\Re z_r(u),
  \qquad
  \eta_r(u):=\Im z_r(u),
  \qquad
  d_r(u):=\dist\bigl(q_r(u),S_r\bigr),
\]
\[
  v(u):=\Im m_0(u)=\Im m_r(z_r(u)).
\]
Then
\begin{equation}
  C^{-1}\eta_*\le v(u)\le C\eta_*,
  \qquad
  C^{-1}\eta_*\le \eta_r(u)\le C\eta_*,
  \qquad
  |q_r(u)|\le C.
  \label{eq:geometry-2}
\end{equation}
\end{lemma}

\begin{proof}
Since \(z_t(u)\in B(x+\ri\eta_*,\eta_*/2)\), we have
\begin{equation}
  \frac12\eta_*\le \eta_t(u)\le \frac32\eta_*,
  \qquad
  \left|d_t(u)-\eps\right|\le \frac{\eta_*}{2},
  \label{eq:geometry-1}
\end{equation}
In particular, \(d_t(u)\ge\eps/2\) for all sufficiently large \(n\). By \Cref{c:edge}, ach edge evolves continuously in time. Hence, there exists
\(K=K(\sfT,K_0)>0\) such that
\(
  S_s\subset[-K,K],
  0\le s\le\sfT.
\)
Since \(x\in\mathcal F_t\), this also gives
\(
  |q_t(u)|\le K+\eps+1
\)
for all sufficiently large \(n\). Using \eqref{e:m0=ms},
\[
  v(u)
  =\Im m_t(z_t(u))
  =\int_{\mathbb R}
  \frac{\eta_t(u)}
  {(y-q_t(u))^2+\eta_t(u)^2}\,\mu_t(dy).
\]
For all sufficiently large \(n\), 
therefore we have
\[
  \frac{\eta_t(u)}{(2K+\eps+1)^2+1}
  \le v(u)
  \le \frac{\eta_t(u)}{d_t(u)^2}
  \le \frac{4\eta_t(u)}{\eps^2}.
\]
Thus \(v(u)\asymp\eta_*\) by \eqref{eq:geometry-1}.

Next, \eqref{e:etas} gives
\[
  \eta_r(u)=\eta_t(u)+(t-r)v(u),
  \qquad 0\le r\le t,
\]
which proves the bounds on \(\eta_r(u)\). Moreover,
\[
  q_r(u)=q_t(u)+(t-r)\Re m_t(z_t(u)).
\]
By \Cref{l:STproperty} and \(d_t(u)\ge\eps/2\),
\[
  |m_t(z_t(u))|
  \le \frac{1}{\dist(z_t(u),S_t)}
  \le \frac{2}{\eps}.
\]
This proves the uniform bound on \(q_r(u)\), and hence
\eqref{eq:geometry-2}.
\end{proof}

The next lemma provides control on the distance between the characteristic flows and the forbidden set.
\begin{lemma}
\label{lem:multicut-reference-geometry}
There exists a constant \(c=c(\eps,\sfT,K_0)>0\) such that, for every
\(0\le r\le t\le\sfT\), every \(x\in\mathcal F_t\), and every
\(u\in D_0^{(t,x)}\),
\begin{equation}
  \dist\bigl(\Re z_r(u),S_r\bigr)
  \ge
  \eps-\frac{\eta_*}{2}
  +c\eps(t-r),
  \label{eq:multicut-real-margin}
\end{equation}
provided \(n\) is sufficiently large.
\end{lemma}

\begin{proof}
Fix \(x\in\mathcal F_t\) and \(u\in D_0^{(t,x)}\), and use the notation
in \Cref{l:geometry-clean}. By \eqref{eq:geometry-1},
\begin{equation}\label{eq:terminal-reference-distance}
  d_t(u)\ge\eps-\frac{\eta_*}{2} \ge 3\eps/4.
\end{equation}
We first derive a differential inequality for \(d_r(u)\) when \(d_r(u)>0\). Suppose that \(q_r\) lies in an interior
gap
\(
  \bigl(b_j(r),a_{j+1}(r)\bigr),
\)
and write
\[
  \kappa_r:=q_r-b_j(r),
  \qquad
  \ell_r:=a_{j+1}(r)-q_r,
  \qquad
  d_r=\min\{\kappa_r,\ell_r\}.
\]
By \Cref{c:edge} and \eqref{e:etas},
\[
  q_r'=-\Re m_r(z_r),
  \qquad
  b_j'(r)=-m_r(b_j(r)),
  \qquad
  a_{j+1}'(r)=-m_r(a_{j+1}(r)).
\]
If \(q_r\) lies in the right (w.r.t. left) exterior component of
\(\mathbb R\setminus S_r\), then
\[
  d_r=q_r-b_p(r)\quad (w.r.t. d_r=a_1(r)-q_r),
\]
and the computation for \(d_r\) applies similarly. For simplicity, we only treat the interior gap case.

By \eqref{eq:geometry-2} and the uniform support bound above, take  is
\(L=\max\{K,C\}>0\), we have
\[
  S_r\subset[-L,L],
  \qquad
  |q_r|\le L.
\]
Hence, for every real \(v\notin S_r\) lying between \(q_r\) and a nearest
endpoint of \(S_r\),
\begin{equation}
  m_r'(v)
  =
  \int_{\mathbb R}\frac{\mu_r(dy)}{(y-v)^2}
  \ge \frac{1}{4L^2}
  =:c_0.
  \label{eq:uniform-real-derivative-lower-bound}
\end{equation}
Moreover, since
\(
  |y-q_r|\ge d_r,
  \forall y\in S_r,
\)
we have
\begin{align}
  \left|m_r(q_r)-\Re m_r(z_r)\right|
  &=
  \left|
  \int_{\mathbb R}
  \frac{\eta_r^2}
  {(y-q_r)((y-q_r)^2+\eta_r^2)}\,\mu_r(dy)
  \right| \le \frac{\eta_r^2}{d_r^3}.
  \label{eq:real-complex-stieltjes-error-distance}
\end{align}
Therefore, when \(d_r=\kappa_r\), then
\begin{align*}
  \partial_r\kappa_r
  &=
  m_r(b_j(r))-\Re m_r(z_r)\\
  &=
  -\int_{b_j(r)}^{q_r}m_r'(v)\,dv
  +m_r(q_r)-\Re m_r(z_r)\\
  &\le
  -c_0d_r+\frac{\eta_r^2}{d_r^3}.
\end{align*}
Similarly, if \(d_r=\ell_r\),
\begin{align*}
  \partial_r\ell_r
  &=
  \Re m_r(z_r)-m_r(a_{j+1}(r))\\
  &=
  -\int_{q_r}^{a_{j+1}(r)}m_r'(v)\,dv
  +\Re m_r(z_r)-m_r(q_r)\\
  &\le
  -c_0d_r+\frac{\eta_r^2}{d_r^3}.
\end{align*}
Thus
\begin{equation}
  D^+d_r
  \le
  -c_0d_r+\frac{\eta_r^2}{d_r^3}, \qquad \text{with } D^+d_r
  :=
  \limsup_{h\downarrow0}
  \frac{d_{r+h}-d_r}{h},
  \label{eq:multicut-distance-differential-inequality}
\end{equation}
Let
\[
  r_*:=
  \inf\left\{
  s\in[0,t]:
  d_r\ge\frac{\eps}{2}
  \text{ for every }r\in[s,t]
  \right\}.
\]
By the continuity of \(q_r\) and edges \(S_r\),
\(r\mapsto d_r\) is continuous. Hence
\eqref{eq:terminal-reference-distance} shows that the set above is
nonempty. On \([r_*,t]\), \Cref{l:geometry-clean} gives
\(\eta_r\le C\eta_*\), and therefore, for all sufficiently large \(n\),
\[
  \frac{\eta_r^2}{d_r^3}
  \le
  \frac{8C^2\eta_*^2}{\eps^3}
  \le
  \frac{c_0}{2}d_r.
\]
Thus \eqref{eq:multicut-distance-differential-inequality} yields
\begin{equation}
  D^+d_r
  \le
  -\frac{c_0}{2}d_r
  \qquad\text{on }[r_*,t].
  \label{eq:reference-distance-dissipation}
\end{equation}
Consequently, it gives
\begin{equation}
  d_r
  \ge
  \exp\left(\frac{c_0}{2}(t-r)\right)d_t,
  \qquad
  r_*\le r\le t.
  \label{eq:reference-distance-exponential-growth}
\end{equation}
If \(r_*>0\), then \eqref{eq:reference-distance-exponential-growth}
and \eqref{eq:terminal-reference-distance} give
\(
  d_{r_*}\ge d_t\ge\frac{3\eps}{4}.
\)
By continuity, this is a contradiction. Hence
\(r_*=0\). 
In the end, we have
\begin{align*}
  d_r
  \ge
  \exp\left(\frac{c_0}{2}(t-r)\right)d_t
  \ge
  \exp\left(\frac{c_0}{2}(t-r)\right)
  \left(\eps-\frac{\eta_*}{2}\right)\ge
  \eps-\frac{\eta_*}{2}
  +c\eps(t-r)
\end{align*}
for a constant \(c=c(\eps,\sfT,K_0)>0\).
This proves
\eqref{eq:multicut-real-margin}.
\end{proof}

\subsubsection*{Step 2. Deterministic confinement before the stopping time}

The next lemma prevents any particle from crossing the forbidden set
\(
  \mathcal F_t
  =
  \{x\in\mathbb R:\dist(x,S_t)=\eps\}
\)
before the stopping time.

\begin{lemma}
\label{l:beforestop-clean}
For all \(0\le t\le\sigma\), and all sufficiently large \(n\),
\begin{equation}
  \Lambda_n(t)\subset S_t^{[\eps]},\quad  \Lambda_n(t)
  \cap
  \mathcal F_t^{[\eta_*]}=\varnothing.
  \label{eq:confinement-before-stopping}
\end{equation}
\end{lemma}

\begin{proof}
We first exclude particles from a neighborhood of the forbidden set $\mathcal F_t$.
Fix \(0<s\le\sigma\). By the definition of \(\sigma\), and by
continuity at the endpoint \(s=\sigma\),
\[
  \sup_{x\in\mathcal F_s}
  \sup_{w\in B(x+\ri\eta_*,\eta_*/2)}
  \left|\widetilde m_s(w)-m_s(w)\right|
  \le
  \frac{2\delta_n}{n\eta_*}.
\]
Fix \(x\in\mathcal F_s\). We apply
\Cref{c:deterministic_comparison} with
\[
  u=x,
  \qquad
  \eta=\eta_*,
  \qquad
  \mu=\mu_s.
\]
Since
\(
\dist\bigl(u,\supp(\mu_s)\bigr)=\eps,
\)
and \(\eta_*=n^{-\mathfrak c}\le n^{-\fc/2}\eps\) for \(n\) large enough, the assumptions of
\Cref{c:deterministic_comparison} are satisfied (up to replacing \(\delta_n\) there by \(2\delta_n\)).
Hence
\[
\lambda_i(s)\notin [x-\eta_*,\,x+\eta_*],\qquad 1\le i\le n.
\]
Varying $x\in\mathcal F_s$ gives
\begin{equation}
  \Lambda_n(s)
  \cap
  \mathcal F_s^{[\eta_*]}
  =
  \varnothing,
  \qquad
  0<s\le\sigma.
  \label{eq:no-particle-in-forbidden-set}
\end{equation}

For each \(1\le i\le n\) and for $0\le s\le\sigma$, define
\(
  \rho_i(s)
  :=
  \dist\bigl(\lambda_i(s),S_s\bigr).
\)
The particle trajectories \(s\mapsto\lambda_i(s)\) are continuous. Moreover,
\(s\mapsto S_s\) is continuous in the sense that the endpoints may merge or move continuously, see \Cref{r:edge-merging}. Hence
\(s\mapsto\rho_i(s)\) is continuous.
At the initial time,
\(
  \Lambda_n(0)\subset S_0,
\)
and hence
\(  \rho_i(0)=0.\)
Suppose, for a contradiction, that there exist \(t_0\le\sigma\) and
\(1\le i\le n\) such that
\(
  \rho_i(t_0)=\eps.
\)
Then
\[
  \dist(\lambda_i(t_0),S_{t_0})=\eps.
\]
On the other hand,
\[
  \lambda_i(t_0)\in\Lambda_n(t_0),
\]
which contradicts
\eqref{eq:no-particle-in-forbidden-set}, since
\(
  \mathcal F_{t_0}
  \subset
  \mathcal F_{t_0}^{[\eta_*]}.
\)
Therefore
\[
  \rho_i(t)<\eps,
  \qquad
  1\le i\le n,
  \qquad
  0\le t\le\sigma.
\]
This proves \eqref{eq:confinement-before-stopping}.
\end{proof}

\subsubsection*{Step 3. Lattice approximation}

We discretize the initial spectral domains by a lattice in \(\bC\). Set
\[
\mathcal L
:=
\left(
\bigcup_{0\le t\le T}\ \bigcup_{x\in\mathcal F_t}
  D_0^{(t,x)}
\right)\cap n^{-8}(\bZ+\ri \bZ).
\]

\begin{lemma}[Lattice approximation]\label{l:approx-clean}
There exists a constant \(C=C(\eps,\sfT)\) such that for all
\(0\le s\le t\le \sfT\), provided \(n\) is sufficiently large, for every \(w\in \mathcal D_s^{(t,x)}\), there exists
\(
u\in \mathcal L\cap \mathcal D_0^{(t,x)} 
\)
such that
\(
|z_s(u)-w|\le C n^{2\mathfrak c-8}.
\)
\end{lemma}

\begin{proof}
Let \(\zeta:=z_s^{-1}(w)\). Then \(\zeta\in \mathcal D_0^{(t,x)}\), hence there exists
\(u\in \mathcal L\cap \mathcal D_0^{(t,x)}\) such that \(|u-\zeta|\le 2n^{-8}\).

Since from \Cref{l:geometry-clean}, \(\Im u\asymp \eta_*\) on \(\bigcup_t \mathcal D_0^{(t,x)}\), we have
\[
|m_0'(u)|\le \frac{C}{(\Im u)^2}\le C n^{2\mathfrak c}.
\]
Therefore
\[
|\partial_u z_s(u)| = |1-s m_0'(u)| \le 1+\sfT |m_0'(u)|\le C n^{2\mathfrak c}.
\]
Thus \(z_s\) is Lipschitz with constant \(C n^{2\mathfrak c}\) on
\(\bigcup_t \mathcal D_0^{(t,x)}\), and hence
\[
|z_s(u)-w|
=
|z_s(u)-z_s(\zeta)|
\le C n^{2\mathfrak c}|u-\zeta|
\le C n^{2\mathfrak c-8}.
\]
\end{proof}

\subsubsection*{Step 4. Bounds on the error terms along characteristics}

Recall that
\[
\widetilde m_s(z)=\frac1n\sum_{i=1}^n\frac{1}{\lambda_i(s)-z},
\qquad z\in \bH.
\]
By It\^{o}'s formula, along a characteristic \(z_s(u)\) we have
\begin{equation}\label{eq:dmt-along-char-clean}
d\widetilde m_s(z_s(u))
=
\partial_z \widetilde m_s(z_s(u))
\bigl(\widetilde m_s(z_s(u))-m_s(z_s(u))\bigr)\,ds
+dM_s(u)+dR_s(u),
\end{equation}
where
\begin{align}
dM_s(u)
&=
-\sqrt{\frac{2}{\beta n^3}}
\sum_{i=1}^n
\frac{dB_i(s)}{(\lambda_i(s)-z_s(u))^2},\label{eq:defM-clean}\\
dR_s(u)
&=
\frac{2-\beta}{\beta n^2}
\sum_{i=1}^n
\frac{ds}{(\lambda_i(s)-z_s(u))^3}.
\label{eq:defR-clean}
\end{align}

We first prove a lower bound on the distance from the characteristic to the particles.

\begin{lemma}\label{l:distance-clean}
There exists \(c=c(\eps,\sfT)>0\) such that the following holds for all
\(0\le r\le t\le \sfT\), all \(u\in \mathcal D_0^{(t,x)}\), and all \(1\le i\le n\),
on the event \(\{r\le \sigma\}\):
\begin{equation}\label{eq:distance-clean}
|\lambda_i(r)-z_r(u)|^2
\ge
c\bigl(\varepsilon^2(t-r)^2+\eta_*^2\bigr).
\end{equation}
\end{lemma}

\begin{proof}
Suppose $\Re z_r(u)$ lies in the interior gap $(b_j(r),a_{j+1}(r))$.
By \Cref{l:beforestop-clean} and notice $a_{j+1}(r)-b_j(r)\geq 2\varepsilon$,
\[
\operatorname{dist}(\lambda_i(r),[a_{j}(r),b_j(r)] \cup[a_{j+1}(r),b_{j+1}(r)] )<\varepsilon-\eta_*.
\]
On the other hand, \Cref{lem:multicut-reference-geometry} gives
\[
  \operatorname{dist}(\operatorname{Re}z_r(u),S_r)
  \ge\varepsilon-\frac{\eta_*}{2}
       +c_0\varepsilon(t-r).
\]
Therefore,
\[
  |\operatorname{Re}z_r(u)-\lambda_i(r)|
  \ge c_0\varepsilon(t-r)+\frac{\eta_*}{2},
\]
which gives
\eqref{eq:distance-clean}.
\end{proof}

We next bound the martingale and drift terms.

\begin{proposition}[Bounds on the error terms]\label{p:error-bounds-clean}
There exists an event \(\Omega\), measurable with respect to the Brownian paths
\(\{B_i(s)\}_{1\le i\le n,\ 0\le s\le \sfT}\), such that
\[
\mathbb P(\Omega)\ge 1-Ce^{-(\log n)^2},
\]
and the following holds on \(\Omega\).
For every \(t\in[0,\sfT]\), every \(x \in \cF_t\), every \(u\in \mathcal L\cap \mathcal D_0^{(t,x)}\), and every \(0\le s\le t\),
\begin{align}
|M_{s\wedge \sigma}(u)|
&\le \frac{C(\log n)^2}{n\sqrt{\eps\,\eta_*}},
\label{eq:martingale-bound-clean}\\
|R_{s\wedge \sigma}(u)|
&\le \frac{C\log n}{n\eps}.
\label{eq:drift-bound-clean}
\end{align}
\end{proposition}

\begin{proof}
We begin with the martingale term. Its quadratic variation is
\[
\langle M(u)\rangle_{s\wedge \sigma}
=
\frac{2}{\beta n^3}
\sum_{i=1}^n
\int_0^{s\wedge \sigma}
\frac{dr}{|\lambda_i(r)-z_r(u)|^4}.
\]
Using \eqref{eq:distance-clean}, we obtain
\[
\frac{1}{|\lambda_i(r)-z_r(u)|^4}
\le
\frac{C}{\varepsilon^2(t-r)^2+\eta_*^2}\cdot
\frac{1}{|\lambda_i(r)-z_r(u)|^2}.
\]
Therefore
\begin{align*}
\langle M(u)\rangle_s
&\le
\frac{C}{n^2}
\int_0^{s\wedge \sigma}
\frac{1}{\varepsilon^2(t-r)^2+\eta_*^2}
\cdot
\frac{\Im \widetilde m_r(z_r(u))}{\eta_r(u)}\,dr.
\end{align*}
Before the stopping time, by definition of \(\sigma\),
\[
\Im \widetilde m_r(z_r(u))
\le
\Im m_r(z_r(u))+\frac{\delta_n}{n\eta_r(u)}
=
v(u)+\frac{\delta_n}{n\eta_r(u)}.
\]
By \Cref{l:geometry-clean}, \(v(u)\asymp \eta_*\) and \(\eta_r(u)\asymp \eta_*\), hence
\[
\frac{\Im \widetilde m_r(z_r(u))}{\eta_r(u)}\le C
\]
for \(n\) sufficiently large. Thus
\[
\langle M(u)\rangle_{s\wedge \sigma}
\le
\frac{C}{n^2}
\int_0^{s\wedge \sigma} \frac{dr}{\varepsilon^2(t-r)^2+\eta_*^2}
\le
\frac{C}{n^2}\int_0^t \frac{dr}{\eps^2(t-r)^2+\eta_*^2}
\le
\frac{C}{n^2 \eps\, \eta_*}.
\]

Fix \(q=(\log n)^2\). By the Burkholder--Davis--Gundy inequality,
\[
\mathbb E\left[\sup_{0\le s\le t}|M_{s\wedge \sigma}(u)|^{2q}\right]
\le
(Cq)^{2q}
\left(\frac{C}{n^2\eps\,\eta_*}\right)^q.
\]
By Markov's inequality, after enlarging \(C\) if necessary,
\[
\mathbb P\left(
\sup_{0\le s\le t}|M_{s\wedge \sigma}(u)|
>
\frac{C(\log n)^2}{n\sqrt{\eps\,\eta_*}}
\right)
\le e^{-2(\log n)^2}.
\]
Since \(
\bigcup_{0\le t\le T}\ \bigcup_{x\in\mathcal F_t}
  D_0^{(t,x)}
\) is contained in a bounded region of \(\bH\) and the
mesh size is \(n^{-8}\), we have \(|\mathcal L|\le n^{20}\) for \(n\) large enough. Taking a union bound over
\(u\in \mathcal L\), we obtain an event \(\Omega\) such that
\[
\mathbb P(\Omega)\ge 1-Ce^{-(\log n)^2},
\]
and \eqref{eq:martingale-bound-clean} holds for every admissible \(t,u,s\).

We now turn to the deterministic drift term. By \eqref{eq:defR-clean},
\[
|R_{s\wedge \sigma}(u)|
\le
\frac{C}{n^2}
\sum_{i=1}^n
\int_0^{s\wedge \sigma}\frac{dr}{|\lambda_i(r)-z_r(u)|^3}.
\]
Using \eqref{eq:distance-clean} once again, we get
\[
\frac{1}{|\lambda_i(r)-z_r(u)|^3}
\le
\frac{C}{\sqrt{\eps^2(t-r)^2+\eta_*^2}}\cdot
\frac{1}{|\lambda_i(r)-z_r(u)|^2}.
\]
Hence
\[
|R_{s\wedge \sigma}(u)|
\le
\frac{C}{n}
\int_0^{s\wedge \sigma}
\frac{1}{\sqrt{\eps^2(t-r)^2+\eta_*^2}}
\cdot
\frac{\Im \widetilde m_r(z_r(u))}{\eta_r(u)}\,dr
\le
\frac{C}{n}
\int_0^t
\frac{dr}{\sqrt{\eps^2(t-r)^2+\eta_*^2}}
\le
\frac{C\log n}{n\eps},
\]
where as above, we use
\[
\frac{\Im \widetilde m_r(z_r(u))}{\eta_r(u)}\le C.
\]
This proves \eqref{eq:drift-bound-clean}.
\end{proof}

\subsubsection*{Step 5. Proof that \(\sigma=\sfT\)}
We complete the bootstrap argument by showing that \(\sigma=\sfT\), which concludes the proof of \Cref{t:DBM_support}.

\begin{proposition}\label{p:stoptime-clean}
With probability at least \(1-Ce^{-(\log n)^2}\), we have \(\sigma=\sfT\).
\end{proposition}

\begin{proof}
We work on the event $\Omega$ from
\Cref{p:error-bounds-clean}.  Fix $t\in[0,\sfT]$,
$x\in\mathcal F_t$, and $u\in\mathcal L\cap D_0^{(t,x)}$.  Define
\[
G_s(u)
:=
\widetilde m_{s\wedge \sigma}(z_{s\wedge \sigma}(u))
-
m_{s\wedge \sigma}(z_{s\wedge \sigma}(u)),
\qquad 0\le s\le t.
\]
By \eqref{eq:exact-initial-stieltjes}, \(G_0(u)=0\). Integrating \eqref{eq:dmt-along-char-clean}, we obtain
\begin{equation}\label{eq:Gs-integral}
G_s(u)
=
\int_0^{s\wedge \sigma}
\partial_z \widetilde m_r(z_r(u))\,G_r(u)\,dr
+
M_{s\wedge \sigma}(u)+R_{s\wedge \sigma}(u).
\end{equation}

We next bound the coefficient \(\partial_z \widetilde m_r(z_r(u))\). By \Cref{l:STproperty},
\[
\bigl|\partial_z \widetilde m_r(z_r(u))\bigr|
\le
\frac{\Im \widetilde m_r(z_r(u))}{\eta_r(u)}.
\]
Before the stopping time,
\[
\Im \widetilde m_r(z_r(u))
\le
\Im m_r(z_r(u))+\frac{\delta_n}{n\eta_r(u)}
=
v(u)+\frac{\delta_n}{n\eta_r(u)}.
\]
Using \Cref{l:geometry-clean}, we deduce that
\begin{equation}\label{eq:derivative-bound-clean}
\bigl|\partial_z \widetilde m_r(z_r(u))\bigr|\le C,
\qquad 0\le r\le t.
\end{equation}

On the event \(\Omega\) from \Cref{p:error-bounds-clean}, combining
\eqref{eq:Gs-integral}, \eqref{eq:martingale-bound-clean},
\eqref{eq:drift-bound-clean}, and \eqref{eq:derivative-bound-clean}, we obtain
\[
|G_s(u)|
\le
C\int_0^s |G_r(u)|\,dr
+
\frac{C(\log n)^2}{n\sqrt{\eps\,\eta_*}}
+
\frac{C\log n}{n\eps}.
\]
By Gr\"onwall's inequality,
\[
|G_s(u)|
\le
C\left(
\frac{(\log n)^2}{n\sqrt{\eps\,\eta_*}}
+
\frac{\log n}{n\eps}
\right),
\qquad 0\le s\le t.
\]
Since \(\eta_*=n^{-\mathfrak c}\), this gives
\begin{equation}\label{eq:lattice-bound-clean}
|G_s(u)|
\le
\frac{C(\log n)^2}{n\sqrt{\eps\,\eta_*}}
=
C(\log n)^2 n^{-1+\mathfrak c/2}.
\end{equation}

Now let \(w\in \mathcal D_{t\wedge \sigma}^{(t,x)}\). By \Cref{l:approx-clean}, there exists
\(
u\in \mathcal L\cap \mathcal D_{0}^{(t,x)})
\)
such that
\[
|z_{t\wedge \sigma}(u)-w|\le Cn^{2\mathfrak c-8}.
\]
On \(\mathcal D_{t\wedge \sigma}^{(t,x)}\), both \(m_{t\wedge \sigma}\) and
\(\widetilde m_{t\wedge \sigma}\) are Lipschitz with constant at most \(Cn^{2\mathfrak c}\), since
\[
|m_{t\wedge \sigma}'(z)|+|\widetilde m_{t\wedge \sigma}'(z)|
\le \frac{C}{(\Im z)^2}\le Cn^{2\mathfrak c}.
\]
Therefore
\[
|m_{t\wedge \sigma}(w)-m_{t\wedge \sigma}(z_{t\wedge \sigma}(u))|
+
|\widetilde m_{t\wedge \sigma}(w)-\widetilde m_{t\wedge \sigma}(z_{t\wedge \sigma}(u))|
\le Cn^{4\mathfrak c-8}.
\]
Combining this with \eqref{eq:lattice-bound-clean}, we obtain
\[
|\widetilde m_{t\wedge \sigma}(w)-m_{t\wedge \sigma}(w)|
\le
\frac{C(\log n)^2}{n\sqrt{\eps\,\eta_*}} + Cn^{4\mathfrak c-8}.
\]
Since \(\eta_*=n^{-\mathfrak c}\) and we can set \(\delta_n=n^{-\mathfrak c/3}\), we have
\[
\frac{(\log n)^2}{n\sqrt{\eps\,\eta_*}} + n^{4\mathfrak c-8}
=
\oo\!\left(\frac{\delta_n}{n\eta_*}\right).
\]
Hence, for \(n\) sufficiently large,
\[
|\widetilde m_{t\wedge \sigma}(w)-m_{t\wedge \sigma}(w)|
<
\frac{\delta_n}{n\,\Im w},
\qquad w\in \mathcal D_{t\wedge \sigma}^{(t,x)},
\]
because \(\Im w\asymp \eta_*\) on \(\mathcal D_{t\wedge \sigma}^{(t,x)}\). This contradicts the definition of
\(\sigma\), unless \(\sigma=\sfT\).

Therefore, on \(\Omega\), we must have \(\sigma=\sfT\). Since
\(\mathbb P(\Omega)\ge 1-Ce^{-(\log n)^2}\), the proof is complete.
\end{proof}

\begin{proof}[Proof of \Cref{t:DBM_support}]
On the event $\{\sigma=\sf T\}$, \Cref{l:beforestop-clean} yields
\[
  \Lambda_n(t)\subset S_t^{[\varepsilon]},
  \qquad 0\le t\le \sf T.
\]
By \Cref{p:stoptime-clean}, 
this proves the result.
\end{proof}